\documentclass[psamsfonts,fceqn,leqno]{amsart}
\usepackage{mathrsfs,latexsym,amsfonts,amssymb,curves,epic}
\usepackage{color}

\newtheorem{theorem}{Theorem}[section]

\newtheorem{proposition}[theorem]{Proposition}
\newtheorem{lemma}[theorem]{Lemma}
\newtheorem{question}[theorem]{Question}

\theoremstyle{definition}
\newtheorem{definition}[theorem]{Definition}

\def\N{{\mathbb N}}

\begin{document}
\title[The $k_{R}$-property on free topological groups]
{The $k_{R}$-property on free topological groups}

  \author{Fucai Lin}
  \address{(Fucai Lin): School of mathematics and statistics,
  Minnan Normal University, Zhangzhou 363000, P. R. China}
  \email{linfucai2008@aliyun.com; linfucai@mnnu.edu.cn}

\author{Shou Lin}
\address{(Shou Lin): Institute of Mathematics, Ningde Teachers' College, Ningde, Fujian
352100, P. R. China} \email{shoulin60@163.com}

  \author{Chuan Liu}
\address{(Chuan Liu): Department of Mathematics,
Ohio University Zanesville Campus, Zanesville, OH 43701, USA}
\email{liuc1@ohio.edu}

  \thanks{The first author is supported by the NSFC (Nos. 11571158, 11201414, 11471153),
  the Natural Science Foundation of Fujian Province (Nos. 2017J01405, 2016J05014, 2016J01671, 2016J01672) of China.}


  \begin{abstract}
A space $X$ is called a $k_{R}$-space, if
$X$ is Tychonoff and the necessary and sufficient condition for a real-valued
function $f$ on $X$ to be continuous is that the restriction of $f$ on each compact subset is
continuous. In this paper, we mainly discuss the $k_{R}$-property on the free topological groups, and generalize some well-known results of K. Yamada's in the free topological groups.
  \end{abstract}

  \maketitle

\section{Introduction}
Recall that $X$ is called a {\it $k$-space}, if the necessary and sufficient condition for a subset $A$ of $X$ to be closed is that $A\cap C$ is closed for every compact subset $C$. It is well-known that the $k$-property generalizing
metrizability has been studied intensively by topologists and analysts. A space $X$ is called a $k_{R}$-space, if
$X$ is Tychonoff and the necessary and sufficient condition for a real-valued
function $f$ on $X$ to be continuous is that the restriction of $f$ on each compact subset is
continuous. Clearly every Tychonoff $k$-space is a {\it $k_{R}$-space}. The converse
is false. Indeed, for any uncountable cardinal $\kappa$ the power $\mathbb{R}^{\kappa}$ is a $k_{R}$-space but not a $k$-space, see \cite[Theorem 5.6]{No1970} and \cite[Problem 7.J(b)]{K1975}. Now, the $k_{R}$-property has been widely used in the study of topology, analysis and category, see \cite{BG2014, B2016, Bl1977, B1981, L2006, Lin1991, M1973}.

Results of our research will be presented in two separate papers. In the paper \cite{LLL2016}, we mainly extend some well-known results in $k$-spaces to $k_{R}$-spaces, and then seek some applications in the study of free Abelian topological groups. In the current paper, we shall deter the $k_{R}$-property on free topological groups and extend some results of K. Yamada's on free topological groups.

The paper is organized as follows. In Section 2, we introduce the necessary notation and terminologies which are
  used for the rest of the paper. In Section 3, we investigate the
  $k_{R}$-property on free topological groups, and generalize some results of K. Yamada's. In section 4, we pose some interesting questions about $k_{R}$-spaces in the class of free topological groups which are still unknown for us.

\maketitle
\section{Preliminaries}
In this section, we introduce the necessary notation and terminologies. Throughout this paper, all topological spaces are assumed to be
  Tychonoff, unless otherwise is explicitly stated.
  First of all, let $\N$ be the set of all positive
  integers and $\omega$ the first countable order. For a space $X$, we always denote the set of all the non-isolated points by $\mbox{NI}(X)$. For undefined
  notation and terminologies, the reader may refer to \cite{AT2008},
  \cite{E1989}, \cite{Gr1984} and \cite{Lin2015}.

  \medskip
  Let $X$ be a topological space and $A \subseteq X$ be a subset of $X$.
  The \emph{closure} of $A$ in $X$ is denoted by $\overline{A}$. Moreover, $A$ is called \emph{bounded} if every continuous real-valued
  function $f$ defined on $A$ is bounded.
 The space $X$ is called a
  \emph{$k$-space} provided that a subset $C\subseteq X$ is closed in $X$ if
  $C\cap K$ is closed in $K$ for each compact subset $K$ of $X$.
A space $X$ is called a $k_{R}$-space, if
$X$ is Tychonoff and the necessary and sufficient condition for a real-valued
function $f$ on $X$ to be continuous is that the restriction of $f$ on each compact subset is
continuous. Note that every Tychonoff $k$-space is a $k_{R}$-space. A subset $P$ of $X$ is called a
  \emph{sequential neighborhood} of $x \in X$, if each
  sequence converging to $x$ is eventually in $P$. A subset $U$ of
  $X$ is called \emph{sequentially open} if $U$ is a sequential neighborhood of
  each of its points. A subset $F$ of
  $X$ is called \emph{sequentially closed} if $X\setminus F$ is sequentially open. The space $X$ is called a \emph{sequential  space} if each
  sequentially open subset of $X$ is open. The space $X$ is said to be {\it Fr\'{e}chet-Urysohn} if, for
each $x\in \overline{A}\subset X$, there exists a sequence
$\{x_{n}\}$ such that $\{x_{n}\}$ converges to $x$ and $\{x_{n}:
n\in\mathbb{N}\}\subset A$.

  \begin{definition}\cite{B2016}
  Let $X$ be a topological space.

  $\bullet$ A subset $U$ of $X$ is called {\it $\mathbb{R}$-open} if for each point $x\in U$ there is a continuous function $f: X\rightarrow [0, 1]$ such that $f(x)=1$ and $f(X\setminus U)\subset\{0\}$. It is obvious that each $\mathbb{R}$-open is open. The converse is true for the open subsets of Tychonoff spaces.

  $\bullet$ A subset $U$ of $X$ is called a {\it functional neighborhood} of a set $A\subset X$ if there is a continuous function $f: X\rightarrow [0, 1]$ such that $f(A)\subset\{1\}$ and $f(X\setminus U)\subset\{0\}$. If $X$ is normal, then each neighborhood of a closed subset $A\subset X$ is functional.
  \end{definition}

  \begin{definition}
  Let $\lambda$ be a cardinal. An indexed family $\{X_{\alpha}\}_{\alpha\in\lambda}$ of subsets of a topological space $X$ is called

 $\bullet$ {\it point-countable} if for any point $x\in X$ the set $\{\alpha\in\lambda: x\in X_{\alpha}\}$ is countable in $X$;

 $\bullet$ {\it compact-countable} if for any compact subset $K$ in $X$ the set $\{\alpha\in\lambda: K\cap X_{\alpha}\neq\emptyset\}$ is countable in $X$;

  $\bullet$ {\it locally finite} if any point $x\in X$ has a neighborhood $O_{x}\subset X$ such that the set $\{\alpha\in\lambda: O_{x}\cap X_{\alpha}\neq\emptyset\}$ is finite in $X$;

  $\bullet$ {\it compact-finite} in $X$ if for each compact subset $K\subset X$ the set $\{\alpha\in\lambda: K\cap X_{\alpha}\neq\emptyset\}$ is finite in $X$;

  $\bullet$ {\it strongly compact-finite} \cite{B2016} in $X$ if each set $X_{\alpha}$ has an $\mathbb{R}$-open neighborhood $U_{\alpha}\subset X$ such that the family $\{U_{\alpha}\}_{\alpha\in\lambda}$ is compact-finite in $X$;

  $\bullet$ {\it strictly compact-finite} \cite{B2016} in $X$ if each set $X_{\alpha}$ has a functional neighborhood $U_{\alpha}\subset X$ such that the family $\{U_{\alpha}\}_{\alpha\in\lambda}$ is compact-finite in $X$.
  \end{definition}

 \begin{definition}\cite{B2016}
  Let $X$ be a topological space and $\lambda$ be a cardinal. An indexed family $\{F_{\alpha}\}_{\alpha\in\lambda}$ of subsets of a topological space $X$ is called a {\it fan} (more precisely, a $\lambda$-fan) in $X$ if this family is compact-finite but not locally finite in $X$. A fan $\{F_{\alpha}\}_{\alpha\in\lambda}$ is called {\it strong} (resp. {\it strict}) if each set $F_{\alpha}$ has a $\mathbb{R}$-open neighborhood (resp. functional neighborhood) $U_{\alpha}\subset X$ such that the family $\{U_{\alpha}\}_{\alpha\in\lambda}$ is compact-finite in $X$.

  If all the sets $F_{\alpha}$ of a $\lambda$-fan $\{F_{\alpha}\}_{\alpha\in\lambda}$ belong to some fixed family $\mathscr{F}$ of subsets of $X$, then the fan will be called an {\it $\mathscr{F}^{\lambda}$-fan}. In particular, if each $F_{\alpha}$ is closed in $X$, then the fan will be called a {\it Cld$^{\lambda}$-fan}.
  \end{definition}

  Clearly, we have the following implications:

  $$\mbox{strict fan}\ \Rightarrow\ \mbox{strong fan}\ \Rightarrow\ \mbox{fan}.$$

  \medskip
  Let $\mathscr P$ be a family of subsets of a space $X$. Then, $\mathscr P$ is called a {\it $k$-network}
   if for every compact subset $K$ of $X$ and an arbitrary open set
  $U$ containing $K$ in $X$ there is a finite subfamily $\mathscr {P}^{\prime}\subseteq
  \mathscr {P}$ such that $K\subseteq \bigcup\mathscr {P}^{\prime}\subseteq U$. Recall that a space $X$ is an \emph{$\aleph$-space} (resp.
  {\it $\aleph_{0}$-space}) if
  $X$ has a $\sigma$-locally finite (resp. countable) $k$-network. Recall that a space $X$ is said to be
  \emph{La\v{s}nev} if it is the continuous closed image of some metric space.

 \begin{definition}\cite{C1961}
A topological space $X$ is a {\it stratifiable space} if for each open subset $U$ in $X$, one can assign a sequence $\{U_{n}\}_{n=1}^{\infty}$ of open subsets of $X$ such that

(a) $\overline{U_{n}}\subset U$;

(b) $\bigcup_{n=1}^{\infty}U_{n}=U;$

(c) $U_{n}\subset V_{n}$ whenever $U\subset V$.
\end{definition}

{\bf Note:} Each La\v{s}nev space is stratifiable \cite{Gr1984}.

  \medskip
  Let $X$ be a non-empty space. Throughout this paper, $X^{-1}=\{x^{-1}: x\in X\}$, which is just a copy of
  $X$. For every $n\in\mathbb{N}$, the $F_{n}(X)$ denotes the
  subspace of $F(X)$ that consists of all the words of reduced length at
  most $n$ with respect to the free basis $X$. Let $e$ be the neutral element of $F(X)$, that is, the empty
  word. For every $n\in\N$, an element
  $x_{1}x_{2}\cdots x_{n}$ is also called a {\it form} for $(x_{1}, x_{2},
  \cdots, x_{n})\in(X\bigoplus X^{-1}\bigoplus\{e\})^{n}$. The word $g$ is
  called {\it reduced} if it does not contain $e$ or any pair of
  consecutive symbol of the form $xx^{-1}$. It follows
  that if the word $g$ is reduced and non-empty, then it is different
  from the neutral element $e$ of $F(X)$. In particular, each element
  $g\in F(X)$ distinct from the neutral element can be uniquely written
  in the form $g=x_{1}^{\epsilon_{1}}x_{2}^{\epsilon_{2}}\cdots x_{n}^{\epsilon_{n}}$, where
  $n\geq 1$, $\epsilon_{1}\in\{-1, 1\}$, $x_{i}\in X$, and the {\it support}
  of $g=x_{1}^{\epsilon_{1}}x_{2}^{\epsilon_{2}}\cdots x_{n}^{\epsilon_{n}}$ is defined as
  $\mbox{supp}(g)=\{x_{1}, \cdots, x_{n}\}$. Given a subset $K$ of
  $F(X)$, we put $\mbox{supp}(K)=\bigcup_{g\in K}\mbox{supp}(g)$. For every $n\in\mathbb{N}$, let $$i_n: (X\bigoplus X^{-1}
  \bigoplus\{e\})^{n} \to F_n(X)$$ be the natural mapping defined by
  $$i_n(x_1, x_2, ... x_n)= x_1x_2\cdots x_n$$
  for each $(x_1, x_2, \cdots, x_n) \in (X\bigoplus -X
  \bigoplus\{0\})^{n}$.

\maketitle
\section{The $k_{R}$-property on free topological groups}
In this section, we investigate the
  $k_{R}$-property on free topological groups, and generalize some results of K. Yamada's. Recently, T. Banakh in \cite{B2016} proved that $F(X)$ is a $k$-space if $F(X)$ is a $k_{R}$-space for a La\v{s}nev space $X$. Indeed, he obtained this result in the class of more weaker spaces. However, he did not discuss the following quesiton:

\begin{question}
Let $X$ be a space. For some $n\in\omega$, if $F_{n}(X)$ is a $k_{R}$-space, is $F_{n}(X)$ a $k$-space?
\end{question}

First, we give the following Theorem~\ref{KR}, which gives a complementary for the result of T. Banakh's.

 \begin{lemma}\label{k-R and k}
 Let $F(X)$ be a $k_{R}$-space. If each $F_{n}(X)$ is a normal $k$-space, then $F(X)$ is a $k$-space.
 \end{lemma}

 \begin{proof}
 It is well-known that each compact subset of $F(X)$ is contained in some $F_{n}(X)$ \cite[Corollary 7.4.4]{AT2008}. Hence it follows from \cite[Lemma 2]{L2006} that $F(X)$ is a $k$-space.
 \end{proof}

\begin{theorem}\label{KR}
Let $X$ be a paracompact $\sigma$-space. Then $F(X)$ is a $k_{R}$-space and each $F_{n}(X)$ is a $k$-space if and only if $F(X)$ is a $k$-space.
\end{theorem}

\begin{proof}
Since $X$ is a paracompact $\sigma$-space, it follows from \cite[Theorem 7.6.7]{AT2008} that $F(X)$ is also a paracompact $\sigma$-space, hence each $F_{n}(X)$ is normal. Now it is legal to apply Lemma~\ref{k-R and k} and finish the proof.
\end{proof}

Next, we shall show that for arbitrary a metrizable space $X$, the $k_{R}$-property of $F_{8}(X)$ implies that $F(X)$ is a $k$-space, see Theorem~\ref{F8}. We first prove two technic propositions. To prove them, we need the description of a neighborhood base of $e$ in $F(X)$ obtained in \cite{U1991}. Let $$H_{0}(X)=\{h=x_{1}^{\varepsilon_{1}}x_{2}^{\varepsilon_{2}}\cdots x_{2n}^{\varepsilon_{2n}}\in F(X): \sum_{i=1}^{2n}\varepsilon_{i}=0, x_{i}\in X\ \mbox{for}\ i\in\{1, 2, \cdots, n\}, n\in\mathbb{N}\}.$$Obviously, the subset $H_{0}(X)$ is a clopen normal subgroup of $F(X)$. It is easy to see that each $h\in H_{0}(X)$ can be represented as $$h=g_{1}x_{1}^{\varepsilon_{1}}y_{1}^{-\varepsilon_{1}}g_{1}^{-1}g_{2}x_{2}^{\varepsilon_{2}}y_{2}^{-\varepsilon_{1}}g_{2}^{-1}\cdots g_{n}x_{n}^{\varepsilon_{n}}y_{n}^{-\varepsilon_{1}}g_{n}^{-1},$$ for some $n\in\mathbb{N}$, where $x_{i}, y_{i}\in X$, $\varepsilon_{i}=\pm 1$ and $g_{i}\in F(X)$ for $i\in\{1, 2, \cdots, n\}$. Let $P(X)$ be the set of all continuous pseudometrics on a space $X$. Then take an arbitrary $r=\{\rho_{g}: g\in F(X)\}\in P(X)^{F(X)}$. Let $$p_{r}(h)=\inf\{\sum_{i=1}^{n}\rho_{g_{i}}(x_{i}, y_{i}): h=g_{1}x_{1}^{\varepsilon_{1}}y_{1}^{-\varepsilon_{1}}g_{1}^{-1}g_{2}x_{2}^{\varepsilon_{2}}y_{2}^{-\varepsilon_{1}}g_{2}^{-1}\cdots g_{n}x_{n}^{\varepsilon_{n}}y_{n}^{-\varepsilon_{1}}g_{n}^{-1}, n\in\mathbb{N}\}$$ for each $h\in H_{0}(X)$. In \cite{U1991},  Uspenskii proved that

\smallskip
(a) $\rho_{r}$ is a continuous on $H_{0}(X)$ and

\smallskip
(b) $\{\{h\in H_{0}(X): p_{r}(h)<\delta\}: r\in P(X)^{F(X)}, \delta>0\}$ is a neighborhood base of $e$ in $F(X)$. Moreover, $p_{r}(e)=0$ for each $r\in P(X)^{F(X)}$.

\begin{proposition}\label{separable}
For a stratifiable $k$-space $X$ if $F_{8}(X)$ is a $k_{R}$-space, then $X$ is separable or discrete.
\end{proposition}

\begin{proof}
Assume on the contrary that $X$ is neither separable nor discrete. Then $X$ contains a space $Y=C\bigoplus D$ as a closed subset, where $C=\{x_{n}: n\in\omega\}\cup\{x\}$ is a convergent sequence with its limit point $\{x\}$ and $D=\{d_{\alpha}: \alpha\in\omega_{1}\}$ is a discrete closed subset of $X$. Since $D$ is a discrete closed subset of $X$, we choose a discrete family $\{O_{\alpha}\}_{\alpha\in\omega_{1}}$ of open subsets such that $d_{\alpha}\in O_{\alpha}$ for each $\alpha\in\omega_{1}$. We may assume that $x_{n}\neq x_{m}$ for arbitrary $n\neq m$ and $C\cap \bigcup_{\alpha\in\omega_{1}}O_{\alpha}=\emptyset$. Since $X$ is stratifiable and $Y$ is closed in $X$, it follows from \cite{U1991} that $F(Y)$ is homeomorphic to a closed subgroup of $F(X)$. Hence $F_{8}(Y)$ is closed subspace of $F_{8}(X)$. Next we shall show that $F_{8}(X)$ contains a strict Cld$^{\omega}$-fan, a contradiction. Indeed, since $F_{8}(X)$ is normal, it suffices to construct a strong Cld$^{\omega}$-fan in $F_{8}(X)$.

For each $\alpha\in\omega_{1}$ choose a function $f_{\alpha}: \omega_{1}\rightarrow\omega$ such that $f_{\alpha}|_{\alpha}: \alpha\rightarrow\omega$ is a bijection. For each $n\in\omega$, let $$F_{n}=\{d_{\beta}^{-1}x^{-1}x_{m}d_{\beta}d_{\alpha}x_{n}x^{-1}d_{\alpha}^{-1}: f_{\alpha}(\beta)=n, \alpha, \beta\in\omega_{1}, m\leq n\}.$$We claim that the family $\{F_{n}: n\in\omega\}$ is a strong Cld$^{\omega}$-fan in $F_{8}(X)$. We divide into the proof by the following three statements.

\smallskip
(1) For each $n\in\omega$, the set $F_{n}$ is closed in $F_{8}(X)$.

\smallskip
Fix an arbitrary $n\in\omega$. It suffices to show that the set $F_{n}$ is closed in $F_{8}(Y)$. Let $Z=\mbox{supp}{F_{n}}.$ Then it is obvious that $Z$ is a closed discrete subspace of $Y$. It follows from \cite{U1991} that $F(Z)$ is homeomorphic to a closed subgroup of $F(Y)$, and thus $F_{8}(Z)$ is a closed subspace of $F_{8}(Y)$. Since $F(Z)$ is discrete and $F_{n}\subset F_{8}(Z)$, the set $F_{n}$ is closed in $F_{8}(Y)$ (and thus closed in $F(X)$).

\smallskip
(2) The family $\{F_{n}: n\in\omega\}$ is strong compact-finite in $F_{8}(X)$.

\smallskip
By induction, choose two families of open neighborhoods $\{W_{n}\}_{n\in\omega}$ and $\{V_{n}\}_{n\in\omega}$ in $X$ satisfy the following conditions:

 \smallskip
(a) for each $n\in\omega$, $x_{n}\in W_{n}$;

\smallskip
(b) for each $n\in\omega$, $x\in V_{n}$  and $V_{n+1}\subset V_{n}$;

\smallskip
(c) for each $n\in\omega$, we have $W_{n}\cap (C\cup D)=\{x_{n}\}$, $V_{n}\cap (D\cup W_{n})=\emptyset$ and $V_{n}\cap C\subset C_{n}$, where $C_{n}=\{x_{m}: m>n\}\cup\{x\}$;

\smallskip
(d) $V_{1}\cap \bigcup_{\alpha\in\omega_{1}}O_{\alpha}=\emptyset$ and $W_{n}\cap \bigcup_{\alpha\in\omega_{1}}O_{\alpha}=\emptyset$ for each $n\in\omega$.

For each $n\in\omega$, let $$U_{n}=\bigcup\{O_{\beta}^{-1}V_{n}^{-1}W_{m}O_{\beta}O_{\alpha}W_{n}V_{n}^{-1}O_{\alpha}^{-1}: f_{\alpha}(\beta)=n, \alpha, \beta\in\omega_{1}, m\leq n\}.$$Obviously, each $U_{n}$ contains in $F_{8}(X)\setminus F_{7}(X)$, and since $F_{8}(X)\setminus F_{7}(X)$ is open in $F_{8}(X)$, it follows from \cite[Corollary 7.1.19]{AT2008} that each $U_{n}$ is open in $F_{8}(X)$. We claim that the family $\{U_{n}: n\in\omega\}$ is compact-finite in $F_{8}(X)$. Suppose not, there exist a compact subset $K$ in $F_{8}(X)$ and an increasing sequence $\{n_{k}\}$ such that $K\cap U_{n_{k}}\neq\emptyset$ for each $k\in\omega$. Since $X$ is stratifiable, $F(X)$ is paracompact, then the closure of the set $\mbox{supp}(K)$ is compact in $X$. However, for each $k\in\omega$, since $K\cap U_{n_{k}}\neq\emptyset$, there exist $m_{k}\in\omega$, $\alpha_{k}\in \omega_{1}$ and $\beta_{k}\in \omega_{1}$ such that $f_{\alpha_{k}}(\beta_{k})=n_{k}$ and $K\cap O_{\beta_{k}}^{-1}V_{n_{k}}^{-1}W_{m_{k}}O_{\beta_{k}}O_{\alpha_{k}}W_{n_{k}}V_{n_{k}}^{-1}O_{\alpha_{k}}^{-1}\neq\emptyset$, hence $\mbox{supp}(K)\cap O_{\alpha_{k}}\neq\emptyset$ and $\mbox{supp}(K)\cap O_{\beta_{k}}\neq\emptyset$. Therefore, the set $\mbox{supp}(K)$ intersects each element of the family $\{O_{\alpha_{k}}, O_{\beta_{k}}: k\in\omega\}$. Since the set $\{n_{k}: k\in\omega\}$ is infinite and $f_{\alpha_{k}}(\beta_{k})=n_{k}$ for each $k\in\omega$, the set $\{\alpha_{k}, \beta_{k}: k\in\omega\}$ is infinite.
Then $\mbox{supp}(K)$ intersects infinitely many $O_{\alpha}$'s, which is a contradiction with the compactness of $\mbox{supp}(K)$. Therefore, the family $\{U_{n}: n\in\omega\}$ is compact-finite in $F_{8}(X)$. Therefore, the family $\{F_{n}: n\in\omega\}$ is strong compact-finite in $F_{8}(X)$.

\smallskip
(3) The family $\{F_{n}: n\in\omega\}$ is not locally finite at the point $e$ in $F_{8}(Y)$ (and thus not locally finite in $F_{8}(X)$).

\smallskip
Indeed, it suffices to show that $e\in \overline{\bigcup_{n\in\omega}F_{n}}\setminus\bigcup_{n\in\omega}F_{n}$ in $F_{8}(Y)$. For any $\delta>0$ and $r=\{\rho_{g}: g\in F(Y)\}$, we shall show that $$\{h\in F_{8}(Y): p_{r}(h)<\delta\}\cap \bigcup_{n\in\omega}F_{n}\neq\emptyset.$$ Since the sequence $\{x_{n}\}$ converges to $x$ and $\rho_{d_{\alpha}}$ and $\rho_{d_{\alpha}^{-1}}$ are continuous pseudometrics on $Y$ for each $\alpha\in\omega_{1}$, there is $n(\alpha)\in\omega$ such that $\rho_{d_{\alpha}}(x_{n}, x)<\frac{\delta}{2}$ and $\rho_{d_{\alpha}^{-1}}(x_{n}, x)<\frac{\delta}{2}$ for each $n\geq n(\alpha)$. Therefore, there are $n_{0}\in\omega$ and uncountable set $A\subset \omega_{1}$ such that $\rho_{d_{\alpha}}(x_{n}, x)<\frac{\delta}{2}$ and $\rho_{d_{\alpha}^{-1}}(x_{n}, x)<\frac{\delta}{2}$ for each $n\geq n_{0}$ and $\alpha\in A$. Choose $\alpha\in A$ that has infinitely many predecessors in $A$. Since $f_{\alpha}(\alpha\cap A)$ is an infinite set, there exist $m>n_{0}$ and $\beta\in\alpha\cap A$ such that $f_{\alpha}(\beta)=m$. Then the word $$g=d_{\beta}^{-1}x^{-1}x_{n_{0}}d_{\beta}d_{\alpha}x_{f_{\alpha}(\beta)}x^{-1}d_{\alpha}^{-1}\in F_{f_{\alpha}(\beta)}=F_{m}.$$Furthermore, we have $$p_{r}(g)=p_{r}(d_{\beta}^{-1}x^{-1}x_{n_{0}}d_{\beta}d_{\alpha}x_{f_{\alpha}(\beta)}x^{-1}d_{\alpha}^{-1})\leq \rho_{d_{\beta}^{-1}}(x_{n_{0}}, x)+\rho_{d_{\alpha}}(x_{m}, x)<\frac{\delta}{2}+\frac{\delta}{2}=\delta.$$Hence $e\in\overline{\bigcup_{n\in\omega}F_{n}}$. Then The family $\{F_{n}: n\in\omega\}$ is not locally finite at the point $e$ in $F_{8}(Y)$.
\end{proof}

\begin{proposition}\label{p1}
For a metrizable space $X$ if $F_{8}(X)$ is a $k_{R}$-space, then $X$ is locally compact.
\end{proposition}

\begin{proof}
Assume on the contrary that $X$ is not locally compact. Then there exists a closed hedgehog subspace $$J=\{x\}\cup(\bigcup_{n\in\omega}X_{n})\cup\{z_{n}: n\in\omega\}$$ such that

(1) $X_{n}=\{y_{n}\}\cup\{x_{n, j}: j\in\omega\}$ is a closed discrete subset of $J$ for each $n\in\omega$;

(2) $\{z_{n}: n\in\omega\}$ is a closed discrete subset of $J$; and

(3) $\{\{x\}\cup\bigcup_{n\geq k}X_{n}: k\in\omega\}$ is a neighborhood base of $x$ in $J$.

By Proposition~\ref{separable}, the space $X$ is separable. Next we shall show that $F_{8}(X)$ contains a strict Cld$^{\omega}$-fan, which is a contradiction with $F_{8}(X)$ is a $k_{R}$-space. Since $F_{8}(X)$ is an $\aleph_{0}$-space by \cite[Theorem 4.1]{AOP1989}, the subspace $F_{8}(X)$ is normal, hence it suffices to show that $F_{8}(X)$ contains a strong Cld$^{\omega}$-fan. Furthermore, it follows from \cite[Proposition 2.9.2]{B2016} that each compact-finite family of subsets of $X$ is strongly
compact-finite, hence it suffices to show that $F_{8}(X)$ contains a Cld$^{\omega}$-fan. For any $n, j\in\omega$, put $$E_{n, j}=\{z_{n}^{-1}y_{j}^{-1}xz_{n}x_{n, j}\}.$$ Then it is obvious that each $E_{n, j}$ is closed. Furthermore, it follows from the proof of \cite[Proposition 2.1]{Y2005} that $x\in\overline{\bigcup_{n, j\in\omega} E_{n, j}}\setminus\bigcup_{n, j\in\omega} E_{n, j}$, and thus the family $\{E_{n, j}: n, j\in\omega\}$ is not locally finite at the point $x$. Next we claim that the family $\{E_{n, j}: n, j\in\omega\}$ is compact-finite.

Suppose not, there exist a compact subset $K$ and two sequences $\{n_{i}\}$ and $\{j_{i}\}$ such that $K\cap E_{n_{i}, j_{i}}\neq\emptyset$. Then the closure of the set supp($K$) is compact since $F_{8}(X)$ is paracompact. Since the family $\{E_{n, j}: n, j\in\omega\}$ are disjoint each other, one of the sequences $\{n_{i}\}$ and $\{j_{i}\}$ is an infinite set. If $\{n_{i}\}$ is an infinite set, then the closed discrete set $\{z_{n_{i}}: i\in\omega\}$ is contained in $\mbox{supp}(K)$, which is a contradiction since $\overline{\mbox{supp}(K)}$ is compact. If $\{j_{i}\}$ is an infinite set and $\{n_{i}\}$ is a finite set, then there exists $N\in\omega$ such that $\{x_{n_{i}, j_{i}}: i\in\mathbb{N}\}\subset \bigcup_{j<N}X_{j}$. Obviously, the closed discrete set $\{x_{n_{i}, j_{i}}: i\in\mathbb{N}\}$ is an infinite set and contains in supp($K$), which is a contradiction.
\end{proof}

Now we can show one of our main theorems in this paper.

\begin{theorem}\label{F8}
For a metrizable space $X$, the following are equivalent:
\begin{enumerate}
\item $F(X)$ is a $k$-space;
\item $F_{8}(X)$ is a $k$-space;
\item $F_{8}(X)$ is a $k_R$-space;
\item the space $X$ is locally compact separable or discrete.
\end{enumerate}
\end{theorem}

\begin{proof}
The equivalence of (1) and (4) was showed in \cite{Y2005}. It is obvious that (2) $\Rightarrow$ (3). By Propositions~\ref{separable} and~\ref{p1}, we have (3) $\Rightarrow$ (4).
\end{proof}

By Theorem~\ref{F8}, it is natural to ask the following question:

\begin{question}
Let $X$ be a metrizable space. If $F_{n}(X)$ is a $k_{R}$-space for some $n\in\{4, 5, 6, 7\}$, then is $F_{8}(X)$ a $k_{R}$-space?
\end{question}

{\bf Note} For each $n\in\{2, 3\}$, the answer to the above question is negative. Indeed, for an arbitrary metrizable space $X$, since $i_{2}$ is a closed mapping, $F_{2}(X)$ is a Fr\'{e}chet-Urysohn space (and thus a $k$-space). For $n=3$, we have the following Theorem~\ref{F3}. However, for each $n\in\{4, 5, 6, 7\}$, the above question is still unknown for us.

\begin{proposition}\label{p2}
For a metrizable space $X$ if $F_{3}(X)$ is a $k_{R}$-space, then $X$ is locally compact or NI($X$) is compact.
\end{proposition}

\begin{proof}
Assume on the contrary that neither $X$ is locally compact nor the set of all non-isolated points of $X$ is compact. Then $X$ contains a closed subspace $$Y=\{x\}\cup\bigcup_{n\in\omega}X_{n}\oplus\bigoplus_{n\in\omega}C_{n},$$where for every $n\in\omega$

$X_{n}=\{x_{n, i}: i\in\omega\}$ is a closed discrete subset of $X$,

$\{\{x\}\cup\bigcup_{m\geq n}X_{m}: n\in\omega\}$ is a neighborhood base of $x$ in $Y$,

$C_{n}=\{c_{n, i}: i\in\omega\}\cup\{c_{n}\}$ is a convergent sequence with its limit $c_{n}$, and

$C_{n}$ is contained in the open subset $U_{n}$ of $X$ such that the family $\{U_{n}: n\in\omega\}$ is discrete in $X$ and $(\{x\}\cup\bigcup_{m\in\omega}X_{m})\cap \overline{\bigcup_{m\in\omega}U_{m}}=\emptyset$.

In order to obtain a contradiction, we shall construct a strict Cld$^{\omega}$-fan in $F_{3}(X)$. For any $n, i\in\omega$, choose an open neighborhood $O_{n}^{i}$ of the point $x_{n, i}$  in $X$ such that the family $\{O_{n}^{i}: i\in\omega\}$ is discrete, $O_{n}^{i}\cap \overline{\bigcup_{m\in\omega}U_{m}}=\emptyset$ and $O_{n}^{i}\cap (\{x\}\cup\bigcup_{n\in\omega}X_{n})=\{x_{n, i}\}$.

For each $n, i\in\omega$, let
$$E(n, i)=\{g_{n, i}=c_{n}c_{n, i}^{-1}x_{n, i}\}$$ and $$U(n, i)=V_{n}^{i}(W_{n}^{i})^{-1}O_{n}^{i},$$where $V_{n}^{i}$ and $W_{n}^{i}$ are two arbitrary open neighborhoods of $c_{n}$ and $c_{n, i}$ in $X$ respectively such that $V_{n}^{i}\cup W_{n}^{i}\subset U_{n}$ and $V_{n}^{i}\cap W_{n}^{i}=\emptyset$.
Obviously, each $E(n, i)$ is closed and it follows from \cite[Corollary 7.1.19]{AT2008} that $U(n, i)$ is an open neighborhood of $E(n, i)$ for each $n, i\in\omega$. In \cite{Y2005}, the author has showed that $x\in\overline{\bigcup_{n, i}E(n, i)}\setminus\bigcup_{n, i}E(n, i)$, hence the family $\{E(n, i): n, i\in\omega\}$ is not locally finite in $F_{3}(X)$. To complete the proof, it suffices to show that the family $\{U(n, i): n, i\in\omega\}$ is compact-finite in $F_{3}(X)$. Suppose not, there exists a compact subset $K$ and two sequences $\{n_{j}\}$ and $\{i_{j}\}$ such that $K\cap U(n_{j}, i_{j})\neq\emptyset$. Similar to the proof of Proposition~\ref{p1}, we can obtain a contradiction.
\end{proof}

\begin{theorem}\label{F3}
For a metrizable space $X$ the following are equivalent:
\begin{enumerate}
\item $F_{3}(X)$ is a $k$-space;
\item $F_{3}(X)$ is a $k_R$-space;
\item the space $X$ is locally compact or NI($X$) is compact.
\end{enumerate}
\end{theorem}

\begin{proof}
The equivalence of (1) and (3) was showed in \cite{Y2005}. The implication of (1) $\Rightarrow$ (2) is obvious. By Proposition~\ref{p2}, we have (2) $\Rightarrow$ (3).
\end{proof}

The following theorem was proved in \cite{LLL2016}.

\begin{theorem}\label{La}\cite{LLL2016}
Let $X$ be a stratifiable space such that $X^2$ is a $k_R$-space. If $X$ satisfies one of the following conditions, then either $X$ is metrizable or $X$ is the topological sum of $k_\omega$-subspaces.
\begin{enumerate}
\item $X$ is a $k$-space with a compact-countable $k$-network;
\item $X$ is a Fr\'echet-Urysohn space with a point-countable $k$-network.
\end{enumerate}
\end{theorem}

Since $F_{2}(X)$ contains a closed copy of $X\times X$, it follows from Theorem~\ref{La} that we have the following theorem.

\begin{theorem}\label{F2}\
Let $X$ be a stratifiable $k$-space with a compact-countable $k$-network. Then the following are equivalent:
\begin{enumerate}
\item $F_2(X)$ is a $k$-space;
\item $F_2(X)$ is a $k_R$-space;
\item either $X$ is metrizable or $X$ is the topological sum of $k_\omega$-subspaces.
\end{enumerate}
\end{theorem}

The following proposition shows that we can not replace ``$F_{3}(X)$'' with ``$F_{4}(X)$'' in Theorem~\ref{F3}.
First, we recall a special space. Let $$M_3=\bigoplus\{C_\alpha: \alpha<\omega_1\},$$ where $C_\alpha=\{c(\alpha, n): n\in \mathbb{N}\}\cup\{c_\alpha\}$ with $c(\alpha, n)\to c_\alpha$ as $n\rightarrow\infty$ for each $\alpha\in\omega_{1}$.

\begin{proposition}\label{M3}
The subspace $F_{4}(M_{3})$ is not a $k_{R}$-space.
\end{proposition}

\begin{proof}
It suffices to show that $F_{4}(M_{3})$ contains a strict Cld$^{\omega}$-fan. It follows from \cite[Theorem 20.2]{JW1997} that we can find two families $\mathscr{A}=\{A_{\alpha}: \alpha\in\omega_{1}\}$ and $\mathscr{B}=\{B_{\alpha}: \alpha\in\omega_{1}\}$ of infinite subsets of $\omega$ such that

(a) $A_{\alpha}\cap B_{\beta}$ is finite for all $\alpha, \beta\in\omega_{1}$;

(b) for no $A\subset \omega$, all the sets $A_{\alpha}\setminus A$ and $B_{\alpha}\cap A$, $\alpha\in\omega_{1}$ are finite.

For each $n\in\omega$, put $$X_{n}=\{c(\alpha, n)c_{\alpha}^{-1}c(\beta, n)c_{\beta}^{-1}: n\in A_{\alpha}\cap B_{\beta}, \alpha, \beta\in\omega_{1}\}.$$ It suffices to show the following three statements.

\smallskip
(1) The family $\{X_{n}\}$ is strictly compact-finite in $F_{4}(M_{3})$.

\smallskip
Since $M_{3}$ is a La\v{s}nev space, it follows from \cite[Theorem 7.6.7]{AT2008} that $F(M_{3})$ is also a paracompact $\sigma$-space, hence $F_{4}(M_{3})$ is paracompact (and thus normal). Hence it suffices to show that the family $\{X_{n}\}$ is strongly compact-finite in $F_{4}(M_{3})$. For each $\alpha\in\omega_{1}$ and $n\in\omega$, let $C_{\alpha}^{n}=C_{\alpha}\setminus\{c(\alpha, m): m\leq n\},$ and put $$U_{n}=\{c(\alpha, n)x^{-1}c(\beta, n)y^{-1}: n\in A_{\alpha}\cap B_{\beta}, \alpha, \beta\in\omega_{1}, x\in C_{\alpha}^{n}, y\in C_{\beta}^{n}\}.$$
Obviously, each $X_{n}\subset F_{4}(M_{3})\setminus F_{3}(M_{3})$. Since $F_{4}(M_{3})\setminus F_{3}(M_{3})$ is open in $F_{4}(M_{3})$, it follows from \cite[Corollary 7.1.19]{AT2008} that each $U_{n}$ is open in $F_{4}(M_{3})$. We claim that the family $\{U_{n}\}$ is compact-finite in $F_{4}(M_{3})$.
Suppose not, then there exist a compact subset $K$ in $F_{4}(M_{3})$ and a subsequence $\{n_{k}\}$ of $\omega$ such that $K\cap U_{n_{k}}\neq\emptyset$ for each $k\in\omega$. For each $k\in\omega$, choose an arbitrary point $$z_{k}=c(\alpha_{k}, n_{k})x_{k}^{-1}c(\beta_{k}, n_{k})y_{k}^{-1}\in K\cap U_{n_{k}},$$ where $x_{k}\in C_{\alpha_{k}}^{n_{k}}$ and $y_{k}\in C_{\beta_{k}}^{n_{k}}$. Since $F_{4}(M_{3})$ is paracompact, it follows from \cite{AOP1989} that the closure of the set supp($K$) is compact in $M_{3}$. Therefore, there exists $N\in\omega$ such that $$\mbox{supp}(K)\cap \bigcup\{C_{\alpha}: \alpha\in\omega_{1}\setminus\{\alpha_{i}\in\omega_{1}: i\leq N\}\}=\emptyset,$$ that is, supp($K$$)\subset \bigcup_{\alpha\in\{\gamma_{i}\in\omega_{1}: i\leq N\}} C_{\alpha}.$ Since each $z_{k}\in K$, there exists $$\alpha_{k}, \beta_{k}\in\{\alpha_{i}\in\omega_{1}: i\leq N\}$$ such that $A_{\alpha_{k}}\cap B_{\beta_{k}}$ is an infinite set, which is a contradiction with $A_{\alpha}\cap B_{\beta}$ is finite for all $\alpha, \beta<\omega_{1}$.

\smallskip
(2) Each $X_{n}$ is closed in $F_{4}(M_{3})$.

\smallskip
Fix an arbitrary $n\in\omega$. Next we shall show that $X_{n}$ is closed in $F_{4}(M_{3})$. Let $Z=\mbox{supp}(X_{n})$. Then $Z$ is a closed discrete subset of $M_{3}$. Since $M_{3}$ is metriable, it follows from \cite{U1991} that $F(Z)$ is homeomorphic to a closed subgroup of $F(M_{3})$, hence $F_{4}(Z)$ is a closed subspace of $F_{4}(M_{3})$. Since $F(Z)$ is discrete and $X_{n}\subset F_{4}(Z)$, the set $X_{n}$ is closed in $F_{4}(Z)$ (and thus closed in $F_{4}(M_{3})$).

\smallskip
(3) The family $\{X_{n}\}$ is not locally finite at the point $e$ in $F_{4}(M_{3})$.

\smallskip
Indeed, it suffices to show that $e\in\overline{\bigcup_{n\in\omega}X_{n}}\setminus\bigcup_{n\in\omega}X_{n}$. For any $\delta>0$ and $r=\{\rho_{g}: g\in F(M_{3})\}$, we shall show that $$\{h\in F_{4}(M_{3}): p_{r}(h)<\delta\}\cap \bigcup_{n\in\omega}X_{n}\neq\emptyset.$$
since $\rho_{e}$ is continuous, we can choose a function $f: \omega_{1}\rightarrow \omega$ such that $\rho_{e}(c(\alpha, n), c_{\alpha})<\frac{\delta}{2}$ for any $\alpha\in\omega_{1}$ and $n\geq f(\alpha)$.
For each $\alpha<\omega_{1}$, put $A_{\alpha}^{\prime}=\{n\in A_{\alpha}: n\geq f(\alpha)\}$ and $B_{\alpha}^{\prime}=\{n\in B_{\alpha}: n\geq f(\alpha)\}$. By the condition (b) of the families $\mathscr{A}$ and $\mathscr{B}$, it is easy to see that there exist $\alpha, \beta\in\omega_{1}$ such that $A_{\alpha}^{\prime}\cap B_{\beta}^{\prime}\neq\emptyset$. So, choose $n\in A_{\alpha}^{\prime}\cap B_{\beta}^{\prime}$. Then $\rho_{e}(c(\alpha, n), c_{\alpha})<\frac{\delta}{2}$ and $\rho_{e}(c(\beta, n), c_{\beta})<\frac{\delta}{2}$. Let $z=c(\alpha, n)c_{\alpha}^{-1}c(\beta, n)c_{\beta}^{-1}$. Then $z\in X_{n}$ and $$p_{r}(z)\leq\rho_{e}(c(\alpha, n), c_{\alpha})+\rho_{e}(c(\beta, n), c_{\beta})<\frac{\delta}{2}+\frac{\delta}{2}=\delta,$$ hence $z\in \{h\in F_{4}(M_{3}): p_{r}(h)<\delta\}\cap \bigcup_{n\in\omega}X_{n}$.
\end{proof}

\begin{theorem}
Let $X$ be a metrizable space. If $F_{4}(X)$ is a $k_{R}$-space, then NI($X$) is separable.
\end{theorem}

\begin{proof}
Suppose not, then $X$ contains a closed copy of $M_{3}$. Use the same notation in Theorem~\ref{M3}. Since $X$ is metrizable, there exists a discrete family $\{U_{\alpha}\}_{\alpha\in\omega_{1}}$ of open subsets in $X$ such that $C_{\alpha}\subset U_{\alpha}$ for each $\alpha\in\omega_{1}$. For arbitrary $(\alpha, n)\in \omega_{1}\times\omega$, choose open neighborhoods $V_{\alpha, n}$ and $W_{\alpha, n}$ of the point $c(\alpha, n)$ and $c_{\alpha}$ in $X$ respectively such that $V_{\alpha, n}\cap W_{\alpha, n}=\emptyset$ and $V_{\alpha, n}\cup W_{\alpha, n}\subset U_{\alpha}$. For each $n\in\mathbb{N}$, put and $$O_{n}=\bigcup\{V_{\alpha, n}W_{\alpha, n}^{-1}V_{\beta, n}W_{\beta, n}^{-1}: n\in A_{\alpha}\cap B_{\beta}, \alpha, \beta\in\omega_{1}\}.$$

Similar to the proof of Theorem~\ref{M3}, we can show that the family $\{O_{n}\}_{n\in\omega}$ of open subsets is compact-finite in $F_{4}(X)$, hence the family $\{X_{n}\}_{n\in\omega}$ is a strict Cld$^{\omega}$-fan in $F_{4}(X)$, which is a contradiction.
\end{proof}

\section{Open questions}
In this section, we pose some interesting questions about $k_{R}$-spaces in the class of free topological groups, which are still unknown for us.

By Theorems~\ref{F8}, \ref{F3} and~\ref{F2}, it is natural to pose the following question.

\begin{question}
Let $X$ be a metrizable space. For each $n\in\{4, 5, 6, 7\}$, if $F_{n}(X)$ is a $k_{R}$-space, is $F_{n}(X)$ a $k$-space?
\end{question}

In \cite{Y2002}, Yamada also made the following conjecture:

\smallskip
\noindent{\bf Yamada's Conjecture:} The subspace $F_4(X)$ is Fr\'echet-Urysohn if the set of all non-isolated points of a metrizable space $X$ is compact.

Indeed, we know no answer for the following question.

\begin{question}
Is $F_4(X)$ a $k_{R}$-space if the set of all non-isolated points of a metrizable space $X$ is compact?
\end{question}

In particular, we have the following question.

\begin{question}
Let $X=C\bigoplus D$, where $C$ is a non-trivial convergent sequence with its limit point and $D$ is an uncountable discrete space. Is $F_4(X)$ a $k_{R}$-space?
\end{question}

In \cite{BG2014}, the authors showed that each closed subspace of a stratifiable $k_{R}$-space is a $k_{R}$-subspace. However, the following two questions are still open.

\begin{question}
Is each closed subgroup of a $k_{R}$-free topological group $k_{R}$?
\end{question}

\begin{question}
Is each subspace $F_{n}(X)$ of a $k_{R}$-free topological group $k_{R}$?
\end{question}

{\bf Acknowledgements}. The authors wish to thank
the reviewers for careful reading preliminary version of this paper and providing many valuable suggestions.


  \end{document}